\documentclass[12pt,twoside, draft]{article}
\usepackage{ifthen}
\usepackage{amsmath}
\usepackage{amsfonts}
\usepackage{latexsym}
\usepackage{amsthm}
\usepackage{amssymb}

\date{}

\begin{document}


\centerline{}

\centerline{}

\centerline {\Large{\bf Levy's phenomenon for entire functions}}

\centerline{\Large{\bf  of several variables }}

\centerline{}

\centerline{\bf {O. V. Zrum}}

\centerline{}

\centerline{Department of Mechanics and Mathematics,}

\centerline{Ivan Franko National University of L'viv, Ukraine}

\centerline{matstud@franko.lviv.ua}

\centerline{}

\centerline{\bf {A. O. Kuryliak}}

\centerline{}

\centerline{Department of Mechanics and Mathematics,}

\centerline{Ivan Franko National University of L'viv, Ukraine}

\centerline{kurylyak88@gmail.com}

\centerline{}

\centerline{\bf {O. B. Skaskiv}}

\centerline{}

\centerline{Department of Mechanics and Mathematics,}

\centerline{Ivan Franko National University of L'viv, Ukraine}

\centerline{matstud@franko.lviv.ua}

\centerline{}

\newtheorem{Theorem}{\quad Theorem}[section]

\newtheorem{Definition}[Theorem]{\quad Definition}

\newtheorem{Corollary}[Theorem]{\quad Corollary}

\newtheorem{Lemma}[Theorem]{\quad Lemma}

\newtheorem{Example}[Theorem]{\quad Example}

\begin{abstract}
For entire functions $f(z)=\sum_{n=0}^{+\infty}a_nz^n,
z\in {\Bbb C},$ P. L${\rm \acute{e}}$vy (1929) established that
in the classical Wiman's inequality
$M_f(r)\leq\mu_f(r)\times $ $\times(\ln\mu_f(r))^{1/2+\varepsilon},\
\varepsilon>0,$ which holds outside a set of finite logarithmic
measure, the constant $1/2$ can be replaced almost surely in some sense, by $1/4$;
here
 $M_f(r)=\max\{|f(z)|\colon |z|=r\},\ \mu_f(r)=\max\{|a_n|r^n\colon n\geq0\},\
r>0. $ In this paper we prove that the phenomenon discovered by P.~L${\rm\acute{e}}$vy holds also in the case of Wiman's inequality
for entire functions of several variables, which gives an affirmative
answer to the question of A.~A.~Goldberg and M. M. Sheremeta
(1996) on the possibility of this phenomenon.
\end{abstract}



\section{Introduction}

For an entire function of the form
 $$f(z)=\sum^{+\infty}_{n=0}a_nz^n$$
we denote $M_f(r)=\max\{|f(z)|\colon|z|=r\},$ $\mu_f(r)=\max\{|a_n|r^n\colon
 n\geq  0\},$ $r>0.$
 It is well known (\cite{1},
 \cite{2}) that for all nonconstant entire function $f(z)$
 and all $\varepsilon>0$ the following inequality
 \begin{equation}\label{1}
 M_f(r)\leq\mu_f(r){(\ln\mu_f(r))}^{{1}/{2}+\varepsilon}
 \end{equation}
 holds for $r>1$ outside an exceptional set $E_f(\varepsilon)$ of finite logarithmic measure
 ($\int_{E_f(\varepsilon)}\frac{dr}r<+\infty$).

In this paper we consider entire functions of $p$ complex variables
  \begin{equation}\label{2}
 f(z)=f(z_1,\ldots,z_p)=\sum_{\|n\|=0}^{+\infty}a_nz^n,
 \end{equation}
 where $z^n=z_1^{n_1}\ldots z_p^{n_p},\ p\in\mathbb{N},\ n=(n_1,\ldots,n_p)\in \mathbb{Z}_+^p,\ \|n\|=\sum_{j=1}^pn_j.$
  For $r=(r_1,\ldots,r_p)\in{{\Bbb  R}}_+^p$ we denote
 \begin{gather*}
 B(r)=\{t\in\mathbb{R}_+^p\colon t_j\geq r_j,\ j\in\{1,\ldots, p\}\},\
  r^{\wedge}=\min\limits_{1\leq i\leq p}r_i,\\
M_f(r)=\max \{|f(z)|\colon  |z_1|=r_1,\ldots,|z_p|=r_p\},\\
\mu_f(r)=\max\{|a_n|r_1^{n_1}\ldots r_p^{n_p}\colon  n\in{\Bbb{Z}}_+^p\},\\
\mathfrak{M}_f(r)=\sum_{\|n\|=0}^{+\infty}|a_n|r^n,\ \ln_kx=\underbrace{\ln\ldots\ln}_kx.
 \end{gather*}

By $\Lambda^p$ we denote the class of entire functions of form (\ref{2})
such that $\frac{\partial}{\partial z_j}f(z)\not\equiv0$ in $\mathbb{C}^p$ for any $j\in\{1,\ldots, p\}.$
We say that a subset $E$ of $\mathbb{R}^p_+$ is a set of asymptotically finite logarithmic measure
(\cite{17}) if $E$ is Lebesgue measurable in $\mathbb{R}^p_+$ and there exists an
$R\in\mathbb{R}^p_+$ such that $E\cap B(R)$ is a set of finite logarithmic measure, i.e.
$$
\idotsint\limits_{E\cap B(R)}\prod_{j=1}^p\frac{dr_j}{r_j}<+\infty.
$$

For entire functions of the form (\ref{2}) analogues of inequality (\ref{1})
can be found in \cite{3, 5, 6, 17}. In particular, the following statement is proved in \cite{17}.

 \begin{Theorem}[\cite{17}]\sl Let $f\in\Lambda^p$ and $\delta>0.$
 \begin{itemize}
   \item[a)] Then there exist $R\in\mathbb{R}^p$ and a subset $E$ of $B(R)$
   of finite logarithmic measure such that for $r\in B(R)\backslash E$ we have
   \begin{equation}\label{3}
     \mathfrak{M}_f(r)\leq\mu_f(r)\Bigl(\prod_{i=1}^p\ln^{p-1}r_i\cdot\ln^p\mu_f(r)\Bigl)^{1/2+\delta}.
   \end{equation}
   \item[b)] If for some $\alpha\in\mathbb{R}^p_+$ we have
   $$
   \mathfrak{M}(r)\geq\exp(r^{\alpha})=\exp(r_1^{\alpha_1}\ldots r_p^{\alpha_p}), \mbox{ as } r^{\wedge}\to+\infty
   $$
   or more generally, for each $\beta>0$
   \begin{equation}\label{4}
   \idotsint\limits_{B(S)}\frac{\prod\limits_{i=1}^pdr_i}{r_1r_2\ldots r_p\ln^{\beta}\mathfrak{M}_f(r)}<+\infty,
   \mbox{ as } S^{\wedge}\to+\infty,
   \end{equation}
   then there exist $R\in\mathbb{R}^p$ and a subset $E$ of $B(R)$
   of finite logarithmic measure such that for $r\in B(R)\backslash E$ we have
   $$
   \mathfrak{M}_f(r)\leq\mu_f(r)\ln^{p/2+\delta}\mu_f(r).
   $$
 \end{itemize}
 \end{Theorem}

\section{Wiman's type inequality for random entire functions of several variables}
   Let $\Omega=[0,1]$ and $P$ be the Lebesgue measure on $\mathbb{R}.$
   We consider the Steinhaus probability space $(\Omega,{\cal A},P)$ where
 $\cal{A}$ is the $\sigma$-algebra of Lebesgue measurable subsets of $\Omega$.
 Let $(\xi_n(\omega))$ be some sequence of random variables defined in this space.

 By $K(f,X)$ we denote the class of random entire functions of the form
  \begin{equation*}
    f(z,t)=\sum_{n=0}^{+\infty} a_nX_n(t)z^n.
  \end{equation*}

   In the case when
 $\mathcal{R}=(X_n(t))$ is the Rademacher sequence P. Levy (\cite{8}) proved
 that for any entire function we can replace  the constant 1/2 by 1/4 in the inequality (\ref{1}) almost surely in the class $K(f,\mathcal{R})$.
 Later P.~Erd\H{o}s and A.~R$\acute{{\rm e}}$\-nyi~(\cite{9})
 proved the same result for the class  $K(f,H),$ where $H=(e^{2\pi
 i\omega_n(t)}),$ $(\omega_n(t))$ is a sequence of independent uniformly distributed
 random variables on $[0,1]$. This statement is true also for any class $K(f,X),$ where $X=(X_n(t))$
 is multiplicative system (MS) uniformly bounded by the number 1.
 That is for all $n\in\mathbb{N}$ and $t\in[0,1]$ we have $|X_n(t)|\leq 1$
 and $$(\forall 1\leq i_1 <i_2<\cdots
 <i_k)\colon\ {\bf M}(X_{i_1}X_{i_2}\cdots X_{i_k})=0,$$
 where ${\bf M}\xi$ is the expected value of a random variable $\xi$ (\cite{fil1}--\cite{fil2}).

  In the spring of 1996 during the report of P. V. Filevych at the Lviv seminar
  of the theory of analytic functions professors
 A.~A.~Goldberg and M.~M.~Sheremeta posed the following question (see \cite{ska zrum ms}).
 Does Levy's effect take place for analogues of Wiman's
 inequality for entire functions of several complex variables?

  In the papers \cite{ska zrum ms}--\cite{ska zrum ntsh} we have found an affirmative
  answer to this question for Fenton's inequality (\cite{4})
  for entire functions of two complex variables.

 In this paper we will give answer to this question for Wiman's type inequality from \cite{17} for entire functions
 of several complex variables.

 Let $Z=(Z_{n}(t))$ be a complex sequence of random variables $Z_{n}(t)=X_{n}(t)+iY_{n}(t)$ such that both $X=(X_{n}(t))$ and $Y=(Y_{n}(t))$ are MS, and $K(f,Z)$ the class of random entire functions of the form
 $$f(z,t)=\sum^{+\infty}_{\|n\|=0}a_{n}Z_{n}(t)z_1^{n_1}\ldots z_p^{n_p}.$$

\begin{Theorem} {
  \sl Let $Z=(Z_{n}(t))$ be a MS uniformly bounded by the number 1, $\delta>0,$ $f\in\Lambda^p.$
  \begin{itemize}
   \item[a)]
  Then
  almost surely in $K(f,Z)$ there exist $R\in\mathbb{R}^p$
  and subset $E^*$ of $B(R)$ of finite logarithmic measure such that
   for all $r\in B(R)\backslash E^*$ we have
   \begin{equation}\label{5}
     M_f(r,t)=\max_{|z|=r}|f(z,t)|\leq\mu_f(r)\Bigl(\prod_{i=1}^p\ln^{p-1}r_i\cdot\ln^p\mu_f(r)\Bigl)^{1/4+\delta}.
   \end{equation}
 \item[b)]
  If for some $\alpha\in\mathbb{R}^p_+$ we have
   $$
   \mathfrak{M}(r,t)\geq\exp(r^{\alpha})=\exp(r_1^{\alpha_1}\ldots r_p^{\alpha_p}) \mbox{ as } r^{\wedge}\to+\infty
   $$
   or more generally, for each $\beta>0$
  inequality (\ref{4}) holds,
   then
  almost surely in $K(f,Z)$ there exist $R\in\mathbb{R}^p$
  and subset $E$ of $B(R)$ of finite logarithmic measure such that
   for all $r\in B(R)\backslash E$ we get
   \begin{equation}\label{star}
   {M}_f(r,t)\leq\mu_f(r)\ln^{p/4+\delta}\mu_f(r).
   \end{equation}
 \end{itemize}}
\end{Theorem}

\begin{Lemma}[\cite{21}]{\sl Let $X=(X_{n}(t))$ be a MS
uniformly bounded by the number 1.
 Then for all $\beta>0$ there exists a constant
  $A_{\beta p}>0,$ which depends on $p$ and $\beta$ only such that for all
  $N\geq N_1(p)=\max\{p,4\pi\}$ and  $\{c_{n}\colon \|n\|\leq N\}\subset{\Bbb C}$
  we have
 \begin{gather}\nonumber
 P\Biggl\{t\colon\max\Biggl\{\Biggl|\sum_{\|n\|=0}^Nc_{n}X_{n}(t)e^{in_1\psi_1}\ldots
 e^{in_p\psi_p}\Biggl|\colon\psi\in[0,2\pi]^p\Biggl\}\geq\\
 \label{6}
 \geq A_{\beta p}S_N\ln^{\frac{1}{2}}N\Biggl\}\leq\frac{1}{N^\beta}
 \end{gather}
  where
  $S_N^2=\sum_{\|n\|=0}^N|c_{n}|^2.$}
\end{Lemma}

  By $H$ we denote the class of function $h\colon \mathbb{R}_+\to\mathbb{R}^p_+$
  such that
$$
\int\limits_1^{+\infty}\ldots\int\limits_1^{+\infty}\frac{du_1\ldots du_p}{h(u)}<+\infty.
$$
We also define for all $i\in\{1,\ldots,p\}$
$$
\partial_i\ln\mathfrak{M}_f(r)=r_i\frac{\partial}{\partial r_i}(\ln\mathfrak{M}_f(r))=\frac1{\mathfrak{M}_f(r)}\sum_{\|n\|=0}^{+\infty}n_i|a_n|r^n.
$$

\begin{Lemma}[\cite{17}] {\sl
Let $h\in H.$ Then there exist $R\in\mathbb{R}_+^p$
and subset $E'$ of $B(R)$ of finite logarithmic measure such that for all
$r\in B(R)\backslash E'$ and $s\in\{1,\ldots,p\}$ we have
\begin{equation}\label{7}
  \partial_s\ln\mathfrak{M}_f(r)\leq h(\ln r_1,\ldots,\ln r_{s-1},\ln\mathfrak{M}_f(r),\ln r_{s+1},\ldots,\ln r_p).
\end{equation}}
\end{Lemma}

\begin{proof}[Proof of Theorem 2.] We can suppose that
  $Z=X=(X_n(t))$ is a MS. Indeed,
 if $Z_n(t)=X_n(t)+iY_n(t)$ then we obtain
 $$
 f(z,t)=\sum_{\|n\|=0}^{+\infty}a_{n}X_{n}(t)z^n+\sum_{\|n\|=0}
^{+\infty}ia_{n}Y_n(t)z^n=f_1(z,t)+f_2(z,t),
$$
   and $\mu (r,f_1)=\mu(r,f_2)=\mu(r,f)$ for all
 $r\in{\Bbb R}_+^p.$   Then from inequality (\ref{5}) we obtain for $r\in B(R)\backslash (E_1\cup E_2)$
 almost surely in $K(f,Z)$
 $$M_{f_j}(r,t)\leq\mu_f(r)\Bigl(\prod_{i=1}^p\ln^{p-1}r_i\cdot\ln^p\mu_f(r)\Bigl)^{1/4+\delta_1}
 \quad j\in\{1,2\},\ \delta_1>0.$$

 So, for $r\in B(R)\backslash (E_1\cup E_2)$ almost surely in $K(f,Z)$ we get
 \begin{gather*}
 M_f(r,t)\leq\sqrt{M_{f_1}^2(r,t)+M_{f_2}^2(r,t)}\leq\\
 \leq\sqrt2\mu_f(r)\Bigl(\prod_{i=1}^p\ln^{p-1}r_i\cdot\ln^p\mu_f(r)\Bigl)^{1/4+\delta_1}<
 \mu_f(r)\Bigl(\prod_{i=1}^p\ln^{p-1}r_i\cdot\ln^p\mu_f(r)\Bigl)^{1/4+\delta}.
 \end{gather*}
   It remains to remark that $E_1\cup E_2$ is a set of
   asymptotically finite logarithmic measure.

 For any $j\in\{1,\ldots,p\}$ we have
 \begin{equation}\label{8}
 \lim_{r_j\to+\infty}\mu_f(r_1^0,\ldots,r_{j-1}^0,r_j,r_{j+1}^0,\ldots,r_p^0)=+\infty
 \end{equation}
for fixed $r^0_i>0,\ i\in\{1,\ldots,p\}\backslash\{j\}.$
Indeed, if (\ref{8}) does not hold, then there exists a constant $C>0$ such that
for all $r_j>r_j^*$ we have $\mu_f(r)<C<+\infty.$ Hence, $\#\{n_j\geq1\colon a_n\neq0\}=0$ and $\frac{\partial}{\partial z_j}f(z)\equiv0$
in $\mathbb{C}^p$. So, $f\notin\Lambda^p,$ which gives a contradiction.

For $k\in {\Bbb N}$ we denote $G_k=\{r=(r_1,\ldots,r_p )\in {\Bbb
R}^p_+\colon k\leq \ln\mu_f(r)< k+1\}\cap [1;+\infty)^p.$ Then $G_k\neq\varnothing$ for $k\geq k_0$ and from (\ref{8})
we deduce that for all $k$ the set $G_k$ is a~bounded set. Let  $G_k^+=\bigcup_{j=k}^{+\infty}G_k$ and
$$
h(r)=\prod_{i=1}^pr_i\ln^{1+\delta_1}r_i\in H, \ \delta_1>0.
$$

By Lemma 2.3 there exist $R_j\in\mathbb{R}_+^p$
and a subset $E_j$ of $B(R_j)$ of finite logarithmic measure such that for all
$r\in B(R_j)\backslash E_j$ and $j\in\{1,\ldots,p\}$ we have
\begin{gather*}
\sum_{\|n\|=0}^{+\infty}n_i|a_n|r^n\leq\mathfrak{M} _f(r)
h(\ln r_1,\ldots,\ln r_{s-1},\ln\mathfrak{M}_f(r),\ln r_{s+1},\ldots,\ln r_n)\leq\\
\leq\mathfrak{M} _f(r)\ln\mathfrak{M} _f(r)\ln_2^{1+\delta_1}\mathfrak{M} _f(r)
\prod_{i=1,\ i\neq j}^p\ln r_i\ln_2^{1+\delta_1}r_i.
\end{gather*}

Therefore for $r\in B(R)\backslash(\cup_{i=1}^pE_i)$ we obtain
\begin{gather*}
\sum_{\|n\|=0}^{+\infty}\|n\||a_n|r^n\leq\mathfrak{M} _f(r)\ln\mathfrak{M} _f(r)\ln_2^{1+\delta_1}\mathfrak{M} _f(r)
\sum_{j=1}^p\Biggl(\prod_{i=1,\ i\neq j}^p\ln r_i\ln_2^{1+\delta_1}r_i\Biggl)\leq\\
\leq p\mathfrak{M} _f(r)\ln\mathfrak{M} _f(r)\ln_2^{1+\delta_1}\mathfrak{M} _f(r)
\prod_{i=1}^p\ln r_i\ln_2^{1+\delta_1}r_i,
\end{gather*}
where $$B(R)\subset\Bigl(\bigcap_{j=1}^pB(R_j)\Bigl)\cap[e,+\infty)^p.$$

By Theorem 1.1 we get for  $r\in B(R)\backslash(\bigcup_{i=1}^pE_i)$
\begin{gather*}
  \sum_{\|n\|=0}^{+\infty}\|n\||a_n|r^n\leq p\mu_f(r)\Bigl(\prod_{i=1}^p\ln^{p-1}r_i\cdot\ln^p\mu_f(r)\Bigl)^{1/2+\delta_1}\times\\
  \times \Bigl(\ln\mu_f(r)+\Bigl(\frac12+\delta_1\Bigl)\Bigl((p-1)\sum_{i=1}^p\ln_2r_i
  +p\ln_2\mu_f(r)\Bigl)\Bigl)^{1/2+\delta_1}
  \prod_{i=1}^p\ln r_i\ln_2^{1+\delta_1}r_i.
  \end{gather*}
  Therefore for $\delta_2>2\delta_1$ and $r\in B(R)\backslash(\bigcup_{i=1}^pE_i)$ we obtain
  \begin{gather*}
   \sum_{\|n\|=0}^{+\infty}\|n\||a_n|r^n
   \leq\\
   \leq \mu_f(r)\ln^{p/2+1+\delta_2}\mu_f(r)\prod_{i=1}^p
   \Bigl((\ln r_i)^{p+\delta_2}(\ln_2r_i)^{1+\delta_2}\sum_{i=1}^p\ln_2^{1/2+\delta_2}r_i\Bigl)<\\
   <\mu_f(r)\ln^{p/2+1+\delta_2}\mu_f(r)\prod_{i=1}^p\Bigl(\ln^p r_i \ln_2^2 r_i\Bigl)^{1+\delta_2}.
  \end{gather*}

  So,
  \begin{gather}\nonumber
    \sum_{\|n\|\geq d}|a_n|r^n\leq\sum_{\|n\|\geq d}\frac{\|n\|}{d}|a_n|r^n=
    \frac1{d}\sum_{\|n\|\geq d}\|n\||a_n|r^n\leq\\
    \label{9}
    \leq\frac1{d}\mu_f(r)\ln^{p/2+1+\delta_2}\mu_f(r)\prod_{i=1}^p\Bigl(\ln^p r_i \ln_2^2 r_i\Bigl)^{1+\delta_2}=\mu_f(r),
  \end{gather}
where $$d=d(r)=\ln^{p/2+1+\delta_2}\mu_f(r)\prod_{i=1}^p\Bigl(\ln^p r_i \ln_2^2 r_i\Bigl)^{1+\delta_2}.$$

   Let $G_k^*=G_k\setminus E_{p+1},$
   $$
   E_{p+1}=\bigcup_{i=1}^pE_i\cup E^*\cup \Biggl(\bigcup_{i=1}^{k_0-1}G_i\Biggl).
   $$
    By $I$ we denote the set of integers
  $k\geq k_0$ such that  $G_k^*\neq\varnothing.$ Then
  $\#I=+\infty.$ For $k\in I$ we choose a sequence $r^{(k)}\in G_k^*.$
 Then for all $r\in G_k^*$ we get
  \begin{equation}\label{10}
 \mu_f(r^{(k)})<e^{k+1}\leq e\mu_f(r),\quad
 \mu_f(r)<e^{k+1}<e\mu_f(r^{(k)}),
 \end{equation}
  and also
$
[1;+\infty)^p\setminus E_{p+1}=\bigcup_{k\in I}G_k^*.
$
For $k\in I$ we denote $N_k=[2d_1(r^{(k)})],$ where
 $$
 d_1(r)=\ln^{p/2+1+\delta_2}(e\mu_f(r))\prod_{i=1}^p\Bigl(\ln^p r_i \ln_2^2 r_i\Bigl)^{1+\delta_2},
 $$
  and for     $r\in G_k^*$
 $$
 W_{N_k}(r,t)=\max\left\{\left|\sum_{\|n\|\leq
 N_k}a_{n}r_1^{n_1}\ldots r_p^{n_p}e^{in_1\psi_1+\ldots+in_p\psi_p}X_{n}(t)\right|\colon\psi
 \in[0,2\pi]^p\right\}.
 $$
    For a Lebesgue measurable set $G\subset G_k^*$ and for $k\in I$ we denote
 \begin{equation*}
  \nu_k(G)=\frac{{\rm meas}_p(G)}{{\rm meas}_p(G_k^*)},
 \end{equation*}
where meas$_p$ denotes the Lebesgue measure on $\mathbb{R}^p.$

 Note that  $\nu_k$ is a probability measure defined on the family of Lebesgue measurable subsets of
 $G_k^*.$   Let
 $\Omega=\bigcup_{k\in I}G_k^*$ and $I=\{k_j\colon
 j\geq1\}\subset {\Bbb N},$ where $k_j<k_{j+1},$ $j\geq1.$ For
 Lebesgue measurable subsets  $G$ of $\Omega$  we denote
 \begin{equation}\label{11}
 \nu(G)=\sum_{j=0}^{+\infty}\frac{1}{2^{k_j}}\Big(1-\Big(\frac{1}{2}\Big)^{k_{j+1}
 -k_j}\Big)\int\limits_G\chi_{G^*_{k_{j+1}}}d\nu_{k_{j+1}},
  \end{equation}
  where $k_0=0$ and  $\chi_A$ is characteristic function of a set
 $A.$ We note
 $$\nu(\Omega)=\sum_{g=0}^{+\infty}\frac{1}{2^{k_j}}\Big(1-\Big(\frac{1}{2}\Big)^
{k_{j+1}-k_j}\Big)\nu(G^*_{k_{j+1}})=\sum_{j=0}^{+\infty}\sum_{s=k_j+1}^
{k_{j+1}}\frac{1}{2^s}=\sum_{s=1}^{+\infty}\frac{1}{2^s}=1.$$
   Thus $\nu$ is a probability measure, which is defined on measurable
   subsets of
     $\Omega.$ On $[0,1]\times\Omega$
 we define the probability measure $P_0=P\otimes\nu,$
 which is a direct product of the probability measures $P$ and $\nu.$
 Now for $k\in I$ we define
 \begin{gather*}
 F_k=\{(t,r)\in
 [0,1]\times\Omega\colon
 W_{N_k}(r,t)>A_{1}S_{N_k}(r)\ln^{1/2}N_k\},\\
 F_k(r)=\{t\in[0,1]\colon
 W_{N_k}(r,t)>A_{1}S_{N_k}(r)\ln^{1/2}N_k\},
 \end{gather*}
  where $S^2_{N_k}(r)=\sum_{\|n\|=0}^{N_k}|a_{n}|^2r^{2n}$
  and $A_p$ is the constant from Lemma 1 with $\beta=1.$ Using Fubini's theorem and Lemma 2.2 with
 $c_n=a_nr^n$ and $\beta=1,$ we get for
 $k\in I$
 $$
 P_0(F_k)=\int\limits_{\Omega}\Bigg(\int\limits_{F_k(r)}dP\Bigg)d\nu=\int\limits_{\Omega}P(F_k(r))d\nu
 \leq\frac{1}{N_k}\nu(\Omega)=\frac{1}{N_k}.
 $$

 Note that $N_k>\ln^{p/2+1}\mu_f(r^{(k)})\geq k^{3/2}.$ Therefore
$\sum_{k\in I}P_0(F_k)\leq$\break $\leq\sum_{k=1}^{+\infty}k^{-3/2}<+\infty.$ By Borel-Cantelli's lemma the infinite quantity
of the events $\{F_k\colon k\in I\}$ may occur with probability zero. So,
$$P_0(F)=1, \quad
F=\bigcup_{s=1}^{+\infty}\bigcap_{k\geq  s,k\in
I}\overline{F_k}\subset [0,1]\times \Omega.$$

Then for any point
$(t,r)\in F$ there exists $k_0=k_0(t,r)$ such that for all
$k\geq k_0,$ $k\in I$ we have
\begin{equation}\label{111}
 W_{N_k}(r,t)\leq A_{1}S_{N_k}(r)\ln^{1/2}N_k.
 \end{equation}

Let $P_j$ be a probability measure defined on
$(\Omega_j,{\mathcal A}_j)$, where ${\mathcal A}_j$ is a
$\sigma$-algebra of subsets $\Omega_j$ ($j\in\{1,\ldots,p\}$) and $P_0$ is the direct product
of probability measures
 $P_1,\ldots,P_p$ defined on
$(\Omega_1\times\ldots\times \Omega_p,\, {\mathcal A}_1\times\ldots\times  {\mathcal
A}_p)$. Here ${\mathcal A}_1\times\ldots\times  {\mathcal
A}_p$ is the
$\sigma$-algebra, which contains all $A_1\times
\ldots\times A_p$, where $A_j\in {\mathcal A}_j$. If $F\subset {\mathcal
A}_1\times\ldots\times  {\mathcal
A}_p$ such that $P_0(F)=1$, then in the case when projection
$$
F_1=\{t_1\in
\Omega_1\colon(\exists (t_2,\ldots,t_p)\in \Omega_2\times\ldots\times\Omega_p)[(t_1,\ldots,t_p)\in F]\}
$$
of the set
$F$ on $\Omega_1$ is $P_1$-measurable we have  $P_1(F_1)=1$.

By $F_{\Omega}$ we denote the projection of $F$ on  $\Omega,$ i.e.
$
F_{\Omega}=\{r\in \Omega\colon(\exists  t)[(t,r)\in F]\}.
$  Then
$\nu(F_\Omega)=1$.

Similarly, the projection of $F$ on $[0,1],$ $F_{[0,1]}=\bigcup_{r\in
\Omega}F(r),$ we obtain $P(F_{[0,1]})=1.$

Let $F^{\wedge}(t)=\{r\in \Omega : (t,r)\in F\}$. By Fubini's theorem
we have
$$
0=\int\limits_X(1-\chi_F)dP_0=\int_0^1\Bigg(\int\limits_\Omega(1-\chi_{F^{\wedge}(t)})
d\nu\Bigg)dP.
$$
So $P$-almost everywhere
$0=\int_\Omega(1-\chi_{F^{\wedge}(t)})d\nu=1-\nu(F^{\wedge}(t))$,
i.e. $\exists\, F_1\subset F_{[0,1]}$, $P(F_1)=1$ such that for all
$t\in F_1$ we get $ \nu(F^{\wedge}(t))=1$.

Indeed, if for some $k\in I,$ $k=k_{j+1}$ we obtain
$\nu_k(F^{\wedge}(t)\cap G_k^*)=q<1,$ then
\begin{gather*}
\nu(F^{\wedge}(t))=\sum_{k\in
 I}\nu_k(F^{\wedge}(t)\cap G_k^*)\leq\sum_{s=0}^{+\infty}\frac{1}{2^{k_s}}
 \Big(1-\Big(\frac{1}{2}\Big)^{k_{s+1}-k_s}
 \Big)-\\
 -(1-q)\frac{1}{2^{k_j}}\Big(1-\Big(\frac{1}{2}\Big)^{k_{j+1}-k_j}\Big)
 =1-(1-q)\frac{1}{2^{k_j}}\Big(1-\Big(\frac{1}{2}\Big)^{k_{j+1}-k_j}\Big)<1.
 \end{gather*}
  For any $t\in F_1$ and $k\in I$ we choose a point $r_0^{(k)}(t)\in
G_k^*$ such that
$$
W_{N_k}(r_0^{(k)}(t),t)\geq\frac{3}{4}M_k(t),\
M_k(t)\stackrel{{\rm def}}{=} \sup\{W_{N_k}(r,t):r\in G_k^*\}.
$$
Then from $\nu_k(F^{\wedge}(t)\cap G_k^*)=1$ for all $k \in I$ it follows that
there exists a point  $r^{(k)}(t)\in G_k^*\cap
F^{\wedge}(t)$ such that
$$
|W_{N_k}(r_0^{(k)}(t),t)-W_{N_k}(r^{(k)}(t),t)|<\frac{1}{4}M_k(t)
$$
  or
  $$
  \frac{3}{4}M_k(t)\leq W_{N_k}(r_0^{(k)}(t),t)\leq
W_{N_k}(r^{(k)}(t),t)+\frac{1}{4}M_k(t).
$$
 Since
$(t,r^{(k)}(t))\in F,$  from inequality (\ref{11}) we obtain
\begin{equation} \label{12} \frac{1}{2}M_k(t)\leq
W_{N_k}(r^{(k)}(t),t)\leq  A_1
 S_{N_k}(r^{(k)}(t))\ln^{1/2}N_k.
\end{equation}
 Now for
$r^{(k)}=r^{(k)}(t)$ we get
$$S_N^2(r^{(k)})\leq\mu_f(r^{(k)})\mathfrak{M} _f(r^{(k)})\leq
\mu_f^2(r^{(k)})\Bigl(\prod_{i=1}^p\ln^{p-1}r^{(k)}_i\cdot\ln^p\mu_f(r^{(k)})\Bigl)^{1/2+\delta}.$$
So, for  $t\in  F_1$ and all  $k\geq k_0(t),$ $k\in I$ we obtain
\begin{equation}\label{13}
S_N(r^{(k)})\leq\mu_f(r^{(k)})\Bigl(\prod_{i=1}^p\ln^{p-1}r^{(k)}_i\cdot\ln^p\mu_f(r^{(k)})\Bigl)^{1/4+\delta/2}.
 \end{equation}

  It follows from (\ref{10}) that $d_1(r^{(k)})\geq d(r)$ for $r\in G_k^*.$ Then for
  $t\in F_1,$ $r\in
F^{\wedge}(t)\cap  G_k^*,$ $k\in I,\ k\geq k_0(t)$ we get
$$
M_f(r,t)\leq
 \sum_{\|n\|\geq
2d_1(r^{(k)})}|a_{n}|r^n+W_{N_k}(r,t)\leq
\sum_{\|n\|\geq2d(r)}|a_{n}|r^n+M_k(t).
$$

  Finally, from (\ref{9}), (\ref{12}), (\ref{13}) for $t\in F_1,$  $r\in
F^{\wedge}(t)\cap G_k^*,$ $k\in I $ and $k\geq k_0(t)$  we deduce
\begin{gather*}
M_f(r^{(k)},t)\leq \mu_f(r^{(k)})+2A_{p}S_{N_k}(r^{(k)})\ln^{1/2}N_k\leq\\
\leq
\mu_f(r^{(k)})+2A_{p}\mu_f(r^{(k)})\Bigl(\prod_{i=1}^p\ln^{p-1}r^{(k)}_i\cdot\ln^p\mu_f(r^{(k)})\Bigl)^{1/4+\delta/2}\times\\
\times\Bigl((p/2+1+\delta_2)\ln_2(e\mu_f(r^{(k)}))+(1+\delta_2)\sum_{i=1}^p(p\ln_2r^{(k)}_i+2\ln_3r^{(k)}_i)\Bigl)^{1/2}.
\end{gather*}
Using inequality (\ref{10})
we get for $t\in F_1,$ $r\in F^{\wedge}(t)\cap G_k^*,$ $k\in I$ and
$k\geq
 k_0(t)$
  \begin{equation}\label{14}
   M_f(r,t)\leq C
\mu_f(r)\Bigl(\prod_{i=1}^p\ln^{p-1}r_i\cdot\ln^p\mu_f(r)\Bigl)^{1/4+3\delta_2/4}.
\end{equation}

We choose $k_1>k_0(t)$ such that for all $r\in G_{k_1}^+$ we have
\begin{equation}\label{15}
C\leq\Bigl(\prod_{i=1}^p\ln^{p-1}r_i\cdot\ln^p\mu_f(r)\Bigl)^{\delta_2/4}.
\end{equation}

Using (\ref{14}) and (\ref{15}) we get that inequality (\ref{5}) holds almost surely
($t\in F_1,$ $P(F_1)=1$) for all
\begin{gather*}
r\in \Bigl(\bigcup_{k\in I}(G_k^*\cap F^{\wedge}(t))\cap
G_{k_1}^+\Bigl)\setminus E^*=\\
= ([1,+\infty)^2\cap
G^+_{k_1})\setminus(E^*\cup G^*\cup E_{p+1})
=[1,+\infty)^2\setminus E_{p+2},
\end{gather*}
where
$$
E_{p+2}=E_{p+1}\cup G^*\cup E^*,\ G^*=\bigcup_{k\in I}(G_k^*\setminus
F^{\wedge}(t)).
$$

  It remains to remark that $\nu(G^*)$ defined in (\ref{11}) satisfies
$\nu(G^*)=$\break $\sum_{k\in I}(\nu_k(G_k^*)-\nu_k(F^{\wedge}(t)))$ $=0.$ Then for all $k\in I$ we obtain
\begin{gather*}
\nu_k(G_k^*\setminus
F^{\wedge}(t))=\frac{{\rm meas}_p(G_k^*\setminus
F^{\wedge}(t))}{{\rm meas}_p(G_k^*)}=0,\\
{\rm
meas}_p(G_k^*\setminus F^{\wedge}(t))=\idotsint\limits_{G_k^*\setminus
F^{\wedge}(t)}\frac{dr_1\ldots dr_p}{r_1\ldots r_p}=0.
\end{gather*}
\vskip-44pt
\end{proof}
\vskip25pt

\section{Some examples}

In this section we prove that the exponent $p/4+\delta$ in the inequality~(\ref{star}) cannot be
replaced by a number smaller than $p/4$.  It follows from such a statement.

\begin{Theorem}
\sl For $f(z)=\exp\{\sum_{i=1}^pz_i\}$ and each $\varepsilon>0$ almost surely
in $K(f,H)$ for $r\in E(\varepsilon)$ we have
$$
M_f(r,t)\geq\mu_f(r)\ln^{p/4-\varepsilon}\mu_f(r),
$$
where $E(\varepsilon)$ is a set of infinite asymptotically logarithmic measure and
$H=\{e^{2\pi i\omega_n}\}, \{\omega_n\}$  is a sequence of independent random variables
uniformly distributed on $[0,1].$
\end{Theorem}

In order to prove this theorem we need such a result.

\begin{Theorem}[\cite{9}]
\sl For the entire function $g(z)=e^z$ and each $\varepsilon>0$
almost surely in $K(g,H)$ we have
\begin{equation}\label{16}
\lim_{r\to+\infty}\frac{M_g(r,t)}{\mu_g(r)\ln^{1/4-\varepsilon}\mu_g(r)}=+\infty.
\end{equation}
\end{Theorem}
\begin{proof}[ Proof of Theorem 3.1]
For the entire function $f(z)=\exp\{\sum_{i=1}^pz_i\}$ we have
$
\ln\mathfrak{M}_f(r)=\sum_{i=1}^pr_i
$ and for each $\beta>0$ we get
$$
\idotsint\limits_{(1,+\infty)^p}\frac{dr_1\ldots dr_p}{r_1\ldots r_p(r_1+\ldots+r_p)^{\beta}}<+\infty.
$$
Therefore the function $f(z)$ satisfies condition (\ref{4}). From (\ref{16}) we have for $r\in(r_0,+\infty)^p$
$$
M_f(r,t)>\mu_f(r)\prod_{i=1}^p\ln^{1/4-\varepsilon}\mu_g(r_i).
$$
Denote $\psi(r)=\ln\mu_g(r).$ Remark that
\begin{gather*}
A_t=\{r\colon r_1=t; r_i\in(t_1,t_2)=(\psi^{-1}(\psi(r_1)/2),\psi^{-1}(2\psi(r_1)))\}\subset\\
\subset
\Bigl\{r\colon \prod_{i=1}^p\psi(r_i)\geq\frac1{2^{p-1}(2p-1)}\Bigl(\sum_{i=1}^p\psi(r_i)\Bigl)^p\Bigl\}.
\end{gather*}
Indeed, if $r\in A_t$ then
\begin{gather*}
\prod_{i=1}^p\psi(r_i)=\psi(r_1)\prod_{i=2}^p\psi(r_i)>\psi(r_1)\prod_{i=2}^p\frac{\psi(r_1)}{2}=\frac{\psi^p(r_1)}{2^{p-1}}=\\
=\frac1{2^{p-1}(2p-1)}(\psi(r_1)+2\psi(r_1)+\ldots+2\psi(r_1))>\frac1{2^{p-1}(2p-1)}\Bigl(\sum_{i=1}^p\psi(r_i)\Bigl)^p.
\end{gather*}
For $\varepsilon_1>p\varepsilon$ and $r\in A=\bigcup_{t=r_0}^{+\infty}A_t$ we obtain
\begin{gather*}
M_f(r,t)>\mu_f(r)\prod_{i=1}^p\ln^{1/4-\varepsilon}\mu_g(r_i)>\mu_f(r)
\frac1{2^{p-1}(2p-1)}\Bigl(\sum_{i=1}^p\ln \mu_g(r_i)\Bigl)^{p/4-p\varepsilon}>\\
>\mu_f(r)\ln^{p/4-\varepsilon_1}\mu_f(r).
\end{gather*}
It remains to prove that the set $A$ has infinite asymptotically logarithmic measure.
It is known (\cite{22}) that $t<\psi^{-1}(t)<3t/2,\ t\to+\infty.$
Therefore,
\begin{gather*}
\mathop{\rm meas}\nolimits_p(A)=\int\limits_{r_0}^{+\infty}\int\limits_{t_1}^{t_2}\ldots\int\limits_{t_1}^{t_2}
\frac{dr_1\ldots dr_p}{r_1\ldots r_p}=\int\limits_{r_0}^{+\infty}\Biggl(\int\limits_{t_1}^{t_2}\frac{dr_2}{r_2}\Biggl)^{p-1}\frac{dr_1}{r_1}=\\
=\int\limits_{r_0}^{+\infty}\Bigl(\ln\psi^{-1}(2\psi(r_1))-\ln\psi^{-1}\Bigl(\frac{\psi(r_1)}2\Bigl)\Bigl)^{p-1}\frac{dr_1}{r_1}>\\>
\int\limits_{r_0}^{+\infty}\Bigl(\ln(2\psi(r_1))-\ln\Bigl(\frac{3\psi(r_1)}4\Bigl)\Bigl)^{p-1}\frac{dr_1}{r_1}=
\ln^{p-1}\frac83\cdot\int\limits_{r_0}^{+\infty}\frac{dr_1}{r_1}=+\infty.
\end{gather*}
\vskip-40pt
\end{proof}


\end{document}